\documentclass[times,review,sort&compress,11pt]{elsarticle}
\usepackage[active]{srcltx}
\usepackage{bbm}
\usepackage{mathrsfs}
\usepackage{tipa}
\usepackage{amsthm}
\usepackage{amssymb}
\usepackage{stackrel}
\usepackage{amsmath,amscd}
\usepackage[numbers]{natbib}
\newtheorem{dl}{Theorem}[section]

\newtheorem{yl}[dl]{Lemma}
\newtheorem{dy}[dl]{Definition}
\newtheorem{tl}[dl]{Corollary}
\newtheorem{lz}[dl]{Example}
\newtheorem{xinzhi}[dl]{Proposition}
\newtheorem{remark}[dl]{Remark}
\usepackage[hidelinks]{hyperref}

\numberwithin{equation}{section}
\newproof{pot1}{Proof of Theorem \ref{mainthnew}}
\newproof{pot2}{Proof of Theorem \ref{mainthnew2}}
\newproof{pot3}{Proof of Theorem \ref{qanalogue}}
\newproof{pot4}{Proof of Theorem \ref{333}}

\def\pf{\noindent {\it Proof.} }

\begin{document}
\title{Further study on elliptic interpolation formulas  for the elliptic Askey-Wilson polynomials and allied identities}
\author{Jin Wang\fnref{fn3,fn4}}
\fntext[fn3]{This  work was supported by NSF of Zhejiang Province (Grant~No.~LQ20A010004).}
\fntext[fn4]{E-mail address: jinwang@zjnu.edu.cn}
\address[Canada]{Department of Mathematics, Zhejiang Normal  University,
Jinhua 321004,~P.~R.~China}
\author{Xinrong Ma\fnref{fn1,fn2}}
\fntext[fn1]{This  work was supported by NSFC grant No. 11971341.}
\fntext[fn2]{E-mail address: xrma@suda.edu.cn.}
\address[P.R.China]{Department of Mathematics, Soochow University, Suzhou 215006, P.R.China}
\begin{abstract}
In this paper,  we introduce the so-called elliptic Askey-Wilson polynomials which are homogeneous polynomials  in two special theta functions. With regard to the significance  of polynomials of such kind, we establish some general elliptic interpolation formulas  by the methods of matrix inversions and of polynomial representations. Furthermore,  we find that the basis of elliptic interpolation space due to Schlosser can be uniquely characterized via the elliptic Askey-Wilson polynomials.  As applications of these elliptic interpolation formulas, we establish some new elliptic function identities, including   an extension of Weierstrass' theta identity, a generalized elliptic Karlsson-Minton type identity, and an elliptic analogue  of Gasper's summation formula for very-well-poised ${}_{6+2m}\phi_{5+2m}$ series.
\end{abstract}
\begin{keyword} elliptic hypergeometric series; elliptic Askey-Wilson polynomial;  theta function; triple product; elliptic interpolation; basis;  matrix inversion; symmetric difference; polynomial representation; summation and transformation; Weierstrass' theta identity.
\vspace{5pt}
\\
{\sl AMS subject classification 2010}: Primary  33D15; Secondary 33E05, 41A05
\end{keyword}
\maketitle
\thispagestyle{empty}
\parskip 7pt
\section{Introduction}
Throughout this paper, we will adopt the standard notation and terminology for basic and elliptic
hypergeometric series  found in the book  \cite{10} by Gasper and Rahman. For instance, the $p$-shifted factorial with $|p|<1$ is defined by
\[(x;p)_\infty:
\:=\:\prod_{n=0}^{\infty}(1-xp^n) \quad\text{and}\quad
(x;p)_n:\:=\:\frac{(x;p)_\infty}{(p^nx;p)_\infty}
\]
for any integer $n$. Its multi-parameter form is compactly abbreviated to
\[(x_1,x_2,\ldots,x_m;p)_n:\:=\:\prod_{k=1}^m(x_k;p)_n.\]
We also need  the modified Jacobi  theta function with
 argument $x\neq 0$ and norm $p$
\begin{align}
\theta(x;p):=(x,p/x;p)_\infty.\label{jacobintheta-def}
\end{align}
The well known Jacobi triple product identity (cf. \cite[(II.28)]{10}) asserts that
\begin{align}
\theta(x;p)=\frac{1}{(p;p)_\infty}\sum_{n=-\infty}^\infty (-1)^np^{n(n-1)/2}x^n.\label{jacobintheta}
\end{align}
By convention, the multi-parameter notation
\[\theta(x_1,x_2,\ldots,x_m;p):\:=\:\prod_{k=1}^m\theta(x_k;p).\]
We also adopt the notation
$$(x;q,p)_n:=\prod_{k=0}^{n-1}\theta(xq^k;p)$$
for the theta $q,p$-shifted factorial (cf. \cite[Eq. (11.2.5)]{10}) together with
\[(x_1,x_2,\ldots,x_m;q,p)_n:\:=\:\prod_{k=1}^m(x_k;q,p)_n.\]
 An ${}_{r+1}E _{r}$ theta   hypergeometric series with base $q$, norm $p$ and  argument $x$   is defined  to be
\begin{align}
{}_{r+1}E _{r}\left[\begin{matrix}a_{1},a_{2},\dots ,a_{r+1}\\ b_{1},b_{2},\dots ,b_{r}\end{matrix}
; q,p; x\right]:=\sum _{n=0} ^{\infty }
\frac{(a_{1},a_2,\ldots
,a_{r+1};q,p)_n}{(q,b_{1},b_2,\ldots,b_{r};q,p)_n}x^n.\label{fpfbasic-1}
\end{align}
Further, if $
a_1q=a_2b_1=a_3b_2=\ldots=a_{r+1}b_{r}$ and $$a_2=qa_1^{1/2},a_3=-qa_1^{1/2},a_4=q(a_1/p)^{1/2},a_5=-q(a_1/p)^{1/2},$$
the ${}_{r+1}E _{r}$ is called \emph{very-well-poised} (VWP).
In particular, when the norm $p=0$ in \eqref{fpfbasic-1} , the ${}_{r+1}E _{r}$ reduces to the usual   basic hypergeometric
 VWP ${}_{r+1}\phi _{r}$ series. Moreover, the ${}_{r+1}V _{r}$ VWP \emph{elliptic} hypergeometric series  (cf. \cite[Eq. (11.2.19)]{10}) is defined to be
\begin{align}
{}_{r+1}V _{r}(a_1;a_6,\ldots,a_{r+1};q,p;x):=\sum _{n=0} ^{\infty }\frac{\theta(a_1q^{2n};p)}{\theta(a_1;p)}
\frac{(a_{1},a_6,\ldots
,a_{r+1};q,p)_n(qx)^n}{(q,a_{1}q/a_6,\ldots,a_{1}q/a_{r+1};q,p)_n}.\label{fpfbasic}
\end{align}
When $x=1$, we denote such series in shorthand notation $${}_{r+1}V _{r}(a_1;a_6,\ldots,a_{r+1};q,p).$$
The ${}_{r+1}V _{r}$ is elliptically balanced if and only if $(qa_6a_7\cdots a_{r+1})^2=(a_1q)^{r-5}.$

As it turns out, function expansions \cite{history} and polynomial interpolations \cite{gas} are two rather old and fundamental subjects in Approximation Theory and Numerical
Analysis. Both  have also received great attention in the context of  various (ordinary, basic, elliptic) hypergeometric series \cite{annaby,sahoo,cooper,ismailadded0,ismail,ito,schlosseradd,schlosser} so far. In this regard,  Ismail \cite{ismailadded0} established a $q$-Taylor theorem  expanding of functions in terms of the Askey-Wilson monomial basis.

\begin{dl}[\mbox{\rm Cf. \cite[Theorem 1.3]{ismailadded0}}]\label{ismail-standon} If $f(x)$ is a polynomial of degree $N$, then
\begin{align}
f(x)=\sum_{k=0}^{N}f_k\phi_k(x;a),\label{1-1}
\end{align}
where
\begin{align*}
f_k&:=\frac{(q-1)^k}{(2a)^k(q;q)_k}q^{-k(k-1)/4}
(\mathcal{D}_{q}^{(k)}f)(x_k),~~ x_k:=(aq^{k/2}+q^{-k/2}/a)/2.
\end{align*}
In the above, the Askey-Wilson monomials $\phi_k(x;a)$ are defined by
$$
\phi_k(x;a):=(ae^{i\theta},ae^{-i\theta};q)_k, \qquad x=\cos \theta
$$
and the Askey-Wilson operator is defined to be
\begin{align*}
\big(\mathcal{D}_{q}f\big)(x):=\frac{\breve{f}\big(q^{1/2}e^{i\theta}\big)-\breve{f}\big(q^{-1/2}e^{i\theta}\big)}
{\iota\big(q^{1/2}e^{i\theta}\big)-\iota\big(q^{-1/2}e^{i\theta}\big)},
\end{align*}
where $\iota(t):=(t+1/t)/2, \breve{f}(t):=f(\iota(t)).$ As usual, for integer $k\geq 1,
\mathcal{D}_{q}^{(k)}=\mathcal{D}_{q}\big(\mathcal{D}_{q}^{(k-1)}\big), $ $\mathcal{D}_{q}^{(1)}=\mathcal{D}_{q}.$
\end{dl}

We should remark that it is this theorem by which  Ismail \cite{ismailadded0} and Stanton \cite{ismailadded2} presented  a unified approach to some basic results from basic hypergeometric series such as the $q$-Pfaff-Saalschutz ${}_3\phi_2$  \cite[(II.12)]{10} and Jackson's VWP ${}_8\phi_7$ summation formulas \cite[(II.22)]{10}, Sears' ${}_4\phi_3$  \cite[(III.15)]{10} and Watson's ${}_8\phi_7$ series transformations \cite[(III.17)]{10}, while some new proofs were given in \cite{ismail} by Ismail  and Simeonov. Also, by the same theorem, Cooper \cite{cooper} proved Watson's  VWP ${}_6\phi_5$ summation formula \cite[(II.20)]{10}. In their paper \cite{ismailadded1} published in 2003, Ismail and Stanton  extended the above polynomial $q$-Taylor theorem to one for entire functions of exponential growth. After that, by combining the above  $q$-Taylor theorem with the idea of polynomial interpolations,  they  \cite{ismailadded2} further put forward  a Lagrange-type interpolation formula for polynomials. The meaning of interpolation is that any polynomial $f(x)$ of degree $N$ is completely determined by its evaluation at the $N+1$ special (interpolation) points, say $\iota(x_k), k=0,1,\ldots, N.$

\begin{dl} [\mbox{\rm Cf. \cite[Theorem 3.4]{ismailadded2}}] With the same notation as above.
For polynomial $f(x)$ of degree at most $N$ and with $x=\cos\theta$, we have the expansion
\begin{align}
\frac{(q,a^2q;q)_N}{(aqe^{i\theta},aqe^{-i\theta};q)_N}f(x)=\sum_{k=0}^N\frac{1-a^2q^{2k}}{1-a^2}
\frac{(a^2,ae^{i\theta},ae^{-i\theta},q^{-N};q)_k}{(q,aqe^{i\theta},aqe^{-i\theta},a^2q^{N+1};q)_k}q^{k(1+N)}f(\iota(aq^k)).
\label{ismailformula}
\end{align}
\end{dl}

One of keys to \eqref{ismailformula} is, as  pointed out in \cite[Section 12.2]{ismailbook} by Ismail and Stanton, that the set $\{\phi_k(x;a)|0\leq k\leq N\}$ forms a basis for the vector space of polynomials, named the {\sl Askey-Wilson basis}.
With regard to studying elliptic hypergeometric series,    Schlosser  \cite{schlosseradd} extended Ismail's $q$-Taylor theorem to the elliptic analogues. One of his main results can be restated as follows.

\begin{dl}[\mbox{\rm Cf. \cite[Theorem 4.2]{schlosseradd}}]\label{schlosser-yoo} Let  $W_c^N$ be the  linear space   spanned by  the set
\begin{align}\bigg\{\frac{g_k(x)}{(cx,c/x;q,p)_k}\bigg\}_{k=0}^N,\label{2-2-22-1206}
\end{align}
where, for integer $k: 0\leq k\leq N,$
\begin{align} \label{2-2-22}
\left\{
   \begin{array}{l}
     g_k(x)=g_k(1/x),\\
    g_k(px)=\displaystyle\frac{1}{p^kx^{2k}}g_k(x).
   \end{array}
 \right.
\end{align}
 Then, for any $f(x)\in W_c^N$,  we have the expansion
\begin{align}
f(x)=\sum_{k=0}^{N}f_k\frac{(ax,a/x;q,p)_k}{(cx,c/x;q,p)_k},\label{2-2}
\end{align}
where
\begin{align}
f_k:=\frac{(-1)^kq^{-k(k-1)/4}\theta(q;p)^k}{(2a)^k(q,c/a,acq^{k-1};q,p)}
\big(\mathcal{D}_{c,q,p}^{(k)}f\big)(aq^{k/2})\nonumber
\end{align}
and the well-poised elliptic Askey-Wilson operator $\mathcal{D}_{c,q,p}$  is defined by
\begin{align*}
\big(\mathcal{D}_{c,q,p}f\big)(x):=2q^{1/2}x\frac{\theta(cxq^{-1/2},cxq^{1/2},
cq^{-1/2}/x,cq^{1/2}/x;p)}{\theta(q,x^2;p)}\big(f(q^{1/2}x)-f(q^{-1/2}x)
\big).
\end{align*}
\end{dl}

 Just like   the \emph{Askey-Wilson basis} which is crucial to Ismail's expansion theorem,  the set
$$
\bigg\{\frac{(ax,a/x;q,p)_k}{(cx,c/x;q,p)_k}\bigg\}_{k=0}^N
$$
is proved to be a basis of $W_c^N$  in \cite[Lemma 4.1]{schlosseradd} by Schlosser and plays an important role in Theorem \ref{schlosser-yoo}.
As further applications of his expansion theorem, Schlosser and Yoo \cite{schlosser} successfully  extended Ismail and Stanton's interpolation formula \eqref{ismailformula} to the following

\begin{dl} [\mbox{\rm Cf. \cite[Theorem 2.6]{schlosser}}]\label{schlosserthm}
If $f(x)$ is in $W_c^N$, then
\begin{align}
\frac{(q,a^2q,cx,c/x;q,p)_N}{(ac,c/a,aqx,aq/x;q,p)_N}f(x)\label{schlosseryoo}\\
=\sum_{k=0}^Nq^k\frac{\theta(a^2q^{2k};p)}{\theta(a^2;p)}
\frac{(a^2,aq/c,ax,a/x,acq^{N},q^{-N};q,p)_k}{(q,ac,aqx,aq/x,aq^{1-N}/c,a^2q^{N+1};q,p)_k}f(aq^k).\nonumber
\end{align}
\end{dl}
In regard to applications of polynomial interpolations to $q$-series, one might not ignore a series of research works \cite{nankaipaper-0,nankaipaper,fu} by Chen, Fu and Lascoux. Indeed, as a discrete analogue of the aforementioned interpolation formulas, it is  proved in \cite{nankaipaper} by  Chen and Fu  that
\begin{dl}[\mbox{\rm Cf. \cite[Theorem 1.1]{nankaipaper}}]\label{mainthm-chen} Let  $N\geq 0$ be  integer and $f(x)$  a theta function satisfying $$f(x)=(px^2/c)^{N}f(c/(px)).$$ Then we have
\begin{align}f(x)=
\sum_{k=0}^{N}C_k
\prod_{i=1}^{k}\theta(x/b_i,c/(b_ix);p)\prod_{i=1}^{N-k}\theta(x/x_i,c/(x_ix);p),\label{4}
\end{align}
where
\begin{align}
C_k:=\frac{f(b_1)}{\prod_{i=1}^{N-k+1}\theta(b_1/x_i,c/(x_ib_1);p)}\delta_{1(b)}\delta_{2(b)}
\cdots\delta_{k(b)}\label{coeffi-old}\\
\times\theta(b_{k+1}/x_{N-k+1},c/(x_{N-k+1}b_{k+1});p).\nonumber
\end{align}
\end{dl}

Recall that the divided difference operator $\delta_{i(a)}$  acting on the left of function $f$ in variables $\{a_n\}_{n\geq 1}$  is defined by
$$
f(\ldots,a_i,a_{i+1},\ldots)\delta_{i(a)}=
\frac{f(\ldots,a_{i+1},a_i,\ldots)-f(\ldots,a_i,a_{i+1},\ldots)}{\theta(a_{i+1}/a_i,c/(a_ia_{i+1});p)}.
$$
A full treatise on the operator $\delta_{n(\bullet)}$ acting on symmetric functions and applications to rational interpolation can be found in \cite{lascoux} due  to Lascoux.

Recently, by means of the $(f,g)$-inversion formula, one of the authors set up in  \cite{wangjinpaper} that

\begin{dl}[{\rm Cf. \cite[Theorem 1.6]{wangjinpaper}}]\label{mainthmyl} Define
\begin{align}\label{PQdef}
  P(x):=\theta(-x^2;p^2)(-p;p)_\infty, ~~~  Q(x):=x\,\theta(-px^2;p^2)(-p;p)_\infty.
\end{align}
Then,  for arbitrary  polynomial $\sum_{k=0}^N\lambda_kx^k$ of degree at most $N$, we have
\begin{align}\sum_{k=0}^{N}\lambda_kP(x)^{k}Q(x)^{N-k}&=
x^N\sum_{k=0}^{N}H_k(N)b_k\theta(x_kb_k,x_k/b_k;p)\label{13}\\
&\quad\times
\prod_{i=0}^{k-1}\theta(b_ix,b_i/x;p)\prod_{i=k+1}^{N}\theta(x_ix,x_i/x;p),\nonumber
\end{align}
where the coefficients
\begin{align}
H_n(N)&:=\sum_{k=0}^{n}\frac{1}{b_k^{N+1}}\frac{\sum_{j=0}^{N}\lambda_j P(b_k)^jQ(b_k)^{N-j}} {\prod_{i=n}^{N}\theta(x_ib_k,x_i/b_k;p)\prod_{i=0,i\neq k}^{n}\theta(b_ib_k,b_i/b_k;p)}.\label{coeffi}
\end{align}
\end{dl}

Especially noteworthy is that one of the most interesting results in \cite{wangjinpaper} reveals  a surprising fact: Weierstrass' theta identity {\rm(cf. \cite[Exercise 2.16(i)]{10})}
\begin{align}
\theta(xa,x/a,bc,b/c;p)-\theta(xc,x/c,ab,b/a;p)=\frac{b}{a}
\theta(xb,x/b,ac,a/c;p)\label{weierstrass}
\end{align}
is equivalent to the almost self-evident algebraic identity
\begin{align}
(x-a)(b-c)+(x-b)(c-a)+(x-c)(a-b)=0.\label{tripleknown}
\end{align}
We refer the reader  to  Koornwinder's paper \cite{koornwinder} for the history and applications of Weierstrass' theta identity to the theory of theta functions.

 Before stating our main theorems, we first need to introduce a new kind of polynomials.
\begin{dy}\label{PQcondition} Let $P(x)$ and $Q(x)$ be given by \eqref{PQdef}.
Any homogeneous polynomial in $P(x)$ and $Q(x)$ of degree $N$ in the form
\begin{align}
\sum_{k=0}^{N}\lambda_kP(x)^{k}Q(x)^{N-k}\label{condition}
\end{align}
is called an elliptic Askey-Wilson polynomial of degree $N$. For brevity, we will use the notation
  $\mathcal{L}_N(P(x),Q(x))$ to denote the set of all homogeneous polynomials in $P(x)$ and $Q(x)$ of degree $N$.
\end{dy}
\begin{remark}\label{remark1.8} The reason why we call \eqref{condition} the elliptic Askey-Wilson polynomial is that
\begin{align}
(ax,a/x;q,p)_k=\prod_{i=0}^{k-1}\theta(aq^ix,aq^i/x;p)
=\frac{1}{x^k}\prod_{i=0}^{k-1}\big(P(aq^i)Q(x)-Q(aq^i)P(x)
\big).\label{pqdecomp}
\end{align}
See Lemma \ref{lemma2} below for the second equality. The special case $p=0$ leads us  to
the Askey-Wilson monomials
$
\phi_k(x;a)=(ax,a/x;q)_k.
$
\end{remark}
In the sense of Definition \ref{PQcondition},  it is easily seen that
\begin{xinzhi}\label{closed} For any integers $m,n\geq 0$, if $f(x)\in\mathcal{L}_{m}(P(x),Q(x))$, $g(x)\in\mathcal{L}_{n}(P(x),Q(x))$, then $f(x)g(x)\in\mathcal{L}_{m+n}(P(x),Q(x)).$
\end{xinzhi}
The main purpose of this paper is, as further development of \cite{wangjinpaper}, to establish a few new interpolation formulas for any elliptic Askey-Wilson polynomials. At first, we can show a new interpolation formula which can be regarded as an explicit version of  Theorem \ref{mainthm-chen} of Chen and Fu.

\begin{dl}\label{mainthmchenfu}For any integer $N\geq 0$, let $\{x_n\}_{n\geq 0}$ and $\{b_n\}_{n\geq 0}$ be two sequences such that $x_k\neq b_i, b_i\neq b_j, 0\leq i\neq j,k\leq N$.  For any  $f(x)\in \mathcal{L}_{N}(P(x),Q(x))$,   we have the expansion
\begin{align}f(x)=x^N
\sum_{k=0}^{N}b_kH_k(N)\prod_{i=0}^{k-1}\theta(b_ix,b_i/x;p)\prod_{i=1}^{N-k}
\theta(x_ix,x_i/x;p),\label{mathdlchenfu}
\end{align}
where, for any $n\leq N$, the coefficient
\begin{align}
H_n(N)&=\sum_{k=0}^{n}\frac{f(b_k)}{b_k^{N+1}}\frac{\theta(x_{N-n+1}b_n,x_{N-n+1}/b_n;p)}
{\theta(x_{N-n+1}b_k,x_{N-n+1}/b_k;p)}\nonumber\\
&\qquad\times\prod_{i=1}^{N-n}\frac{1}{\theta(x_ib_k,x_i/b_k;p)}\prod_{i=0,i\neq k}^{n}\frac{1} {\theta(b_ib_k,b_i/b_k;p)}.\label{coeffi-chenfu-new}
\end{align}
 \end{dl}
Closely related with Theorem \ref{mainthmyl} and Weierstrass' theta identity \eqref{weierstrass} is the following   Lagrange-type interpolation formula.
\begin{dl}\label{mainthmyl-latest} With the same assumption as Theorem \ref{mainthmchenfu}. Then for any $f(x)\in \mathcal{L}_{N}(P(x),Q(x))$, we have
\begin{align}
\frac{f(x)}{x^N}=\sum_{k=0}^{N}\frac{f(b_k)}{b_k^N}\prod_{i=0,i\neq k}^{N}\frac{\theta(b_ix,b_i/x;p)}{\theta(b_ib_k,b_i/b_k;p)}\label{13-13}.
\end{align}
\end{dl}

Furthermore,  using  Theorem  \ref{mainthmchenfu}, we can establish
\begin{dl}\label{mainthm} For any $f(x)\in \mathcal{L}_{N}(P(x),Q(x))$, we have the expansion
 \begin{align}\bigg(\frac{C}{x}\bigg)^N\frac{(q,C^2q;q,p)_{N}}
 {(Cxq,Cq/x;q,p)_{N}}f(x)=\sum_{k=0}^{N}q^{k}\frac{\theta(C^2q^{2k};p)} {\theta(C^2;p)}\frac{(C^2,C/x,Cx,q^{-N};q,p)_{k}} {(q,Cxq,Cq/x,C^2q^{N+1};q,p)_{k}}f\big(Cq^k\big).\label{xe33-3}
\end{align}
\end{dl}
Once taking Proposition \ref{closed} into account,  we
can show an even more general interpolation formula.
\begin{dl}\label{generialized} Assume that $f(x)\in \mathcal{L}_{N_0}(P(x),Q(x))$. For $m$ integers $N_i\geq 0$, let $N=\sum_{i=0}^mN_i$.
Then we have the expansion
 \begin{align}&\bigg(\frac{C}{x}\bigg)^{N_0}\frac{(q,C^2q;q,p)_{N}}{(Cxq,Cq/x;q,p)_{N}}
 \bigg(\prod_{i=1}^m\frac{(A_ix,A_i/x;q,p)_{N_i}}{(A_iC,A_i/C;q,p)_{N_i}}\bigg)~f(x)\label{xe33-3-3}\\
&=\sum_{k=0}^{N}q^{k}\frac{\theta(C^2q^{2k};p)} {\theta(C^2;p)}\frac{(C^2,C/x,Cx,q^{-N};q,p)_{k}} {(q,Cxq,Cq/x,C^2q^{N+1};q,p)_{k}}f(Cq^k)\prod_{i=1}^m\frac{(A_iCq^{N_i},Cq/A_i;q,p)_k}{(Cq^{-N_i+1}/A_i,A_iC;q,p)_k}.\nonumber
\end{align}
\end{dl}
Notice  that there is a subtle difference between Theorem \ref{schlosserthm} and the special case $m=1$ of Theorem \ref{generialized}. It is this difference that inspires us to explore $W_c^N$. Up to this point, we find a new  characteristic of $W_c^N$, which states that $x^Ng_N(x)\in \mathcal{L}_{N}(P(x),Q(x))$  being  an elliptic Askey-Wilson polynomial of degree $N$. For comparison purpose, we mention here that   $g_N(x)$ is referred to as $BC_1$ theta functions of degree one by Rains \cite[Definition 1]{17} and to as $D_N$ theta functions by  Rosengren and Schlosser  \cite[Definition 3.1]{20}, respectively. The reader may consult  loc.cit. for their full  exposition.

\begin{dl}\label{chara} Let $P(x)$ and $Q(x)$ be given by  \eqref{PQdef}, and $g_m(x)$ be such a function that
$$
\frac{g_m(x)}{(cx,c/x;q,p)_m}\in W_c^N~~(m\leq N).
$$
 Then,  there must exist a sequence of complex numbers $\{\lambda_k\}_{k=0}^ N$ being independent of $x$, such that
\begin{align}g_m(x)=\frac{1}{x^m}\bigg(\sum_{k=0}^N\lambda_k P(x)^kQ(x)^{N-k}\bigg)\prod_{k=0}^{N-m-1}\frac{1}{P(cq^{m+k})Q(x)-Q(cq^{m+k})P(x)}.\label{threethreethree}
\end{align}
In particular,
\begin{align}g_N(x)=\frac{1}{x^N}\sum_{k=0}^N\lambda_k P(x)^kQ(x)^{N-k}.\label{oneoneone}
\end{align}
\end{dl}

The rest of our paper is organized as follows. In succeeding section,  some preliminary results about the $(f,g)$-inversion formula and polynomial expansions  are given in details. They are keys to Theorems \ref{mainthmchenfu} and \ref{mainthmyl-latest}. The full proofs of the main theorems are given in Section
 3. Some elliptic function identities
 including an extension of Weierstrass' theta identity and an elliptic analogue  of Gasper's summation formula for VWP ${}_{6+2m}\phi_{5+2m}$ series, as well as  a generalized elliptic Karlsson-Minton type identity are  presented in Section 4.

\section{Preliminaries}
One of our main ingredients for Theorems \ref{mainthmchenfu} and \ref{mainthmyl-latest},  instead of the Askey-Wilson operator $\mathcal{D}_q$ and the divided difference operator $\delta_{n(\bullet)}$, is the technique of matrix inversions (in the sense of \eqref{inversedef}). Among matrix inversions, a typical result is the following
\begin{yl}[The $(f,g)$-inversion formula. {\rm Cf. \cite[Theorem 1.3]{0020}}]\label{fglm}
 Let $A=(A_{n,k})_{n\geq k\geq 0}$ and $B=(B_{n,k})_{n\geq k\geq 0}$ be a pair of
{\sl infinite-dimensional
    lower-triangular}  matrices
with entries given by
\begin{align}
A_{n,k}&=\frac{\prod_{i=k}^{n-1}f(x_i,b_k)}
{\prod_{i=k+1}^{n}g(b_i,b_k)}\label{news1115550}\qquad\mbox{and}\\
B_{n,k}&=
\frac{f(x_k,b_k)}{f(x_n,b_n)}\frac{\prod_{i=k+1}^{n}f(x_i,b_n)}
{\prod_{i=k}^{n-1}g(b_i,b_n)},\label{news1115551}\quad\mbox{respectively},
\end{align}
where $\{x_n\}_{n\geq 0}$ and
$\{b_n\}_{n\geq 0}$ are two arbitrary sequences such that none of the denominators
in the right-hand sides of (\ref{news1115550}) and
(\ref{news1115551}) vanish. Then $A=(A_{n,k})_{n\geq k\geq 0}$ and
$B=(B_{n,k})_{n\geq k\geq 0}$ is a matrix inversion, namely,
  \begin{align}
 \sum_{n\geq i\geq k}A_{n,i}B_{i,k}=\sum_{n\geq i\geq k}B_{n,i}A_{i,k}=\delta_{n,k},\label{inversedef}
\end{align}
where $\delta_{n,k}$ denotes the usual Kronecker delta,
if and
only if for all complex numbers $a,b,c,x$,
 \begin{eqnarray}
g(a,b)f(x,c)+g(b,c)f(x,a)+g(c,a)f(x,b)=0\label{triid}
\end{eqnarray}
with a prior requirement  that  $g(x,y)=-g(y,x).$
\end{yl}

As the earlier work of \cite{0020} displays, the $(f,g)$-inversion formula contains many known inverse relations   useful to the study of $q$-series as special cases.   The reader is referred to \cite{egobook,henrichi} for further details on inverse relations and the classical lagrange inversion formula, and to \cite{0020,0021} for applications of the $(f,g)$-inversion formulas.

For completeness,  we recall here  two  variants of  the $(f,g)$-inversion  formula.

\begin{yl}[{\rm Cf. \cite[Lemma 2.2]{wangjinpaper}}]\label{dl2}
 Let $\{x_n\}_{n\geq 0}$ and $\{b_n\}_{n\geq 0}$
 be arbitrary complex sequences such that $b_n$'s are pairwise
 distinct, $g(x,y)=-g(y,x),f(x,y)$ is subject to $(\ref{triid})$.
  Then the linear system with respect to two sequences $\{F_n\}_{n\geq 0}$ and $\{G_n\}_{n\geq 0}$ \begin{eqnarray}
F_n=\sum_{k=0}^{n}G_k f(x_k,b_k)
\frac{\prod_{i=0}^{k-1}g(b_i,b_n)}{\prod_{i=1}^{k}f(x_i,b_n)}\label{27}
\end{eqnarray}
is equivalent to \begin{eqnarray}
G_n=\sum_{k=0}^{n}F_k\frac{\prod_{i=1}^{n-1} f(x_i,b_k)}
 {\prod_{i=0,i\neq k}^{n}g(b_i,b_k)}.\label{28}
\end{eqnarray}
\end{yl}

As demonstrated clearly in  \cite{wangjinpaper},   the above $(f,g)$-inversion formula  can be reformulated as follows.
\begin{yl}[The $(f,g)$-expansion formula. {\rm Cf. \cite[Lemma 2.3]{wangjinpaper}}] \label{fglemma}  With all conditions as in Lemma \ref{dl2}. If there exists an expansion of the form
\begin{eqnarray}
F(x)=\sum_{k=0}^N G_k f(x_k,b_k)\frac{\prod_{i=0}^{k-1}g(b_i,x)}{\prod_{i=1}^{k}f(x_i,x)},
\label{equfun0}
\end{eqnarray}
where $N$ is either finite or infinite integer,
then for all $n\leq N$, the coefficients
\begin{eqnarray}
G_n=\sum_{k=0}^{n}F(b_k)\frac{\prod_{i=1}^{n-1} f(x_i,b_k)}
 {\prod_{i=0,i\neq k}^{n}g(b_i,b_k)}.\label{280}
\end{eqnarray}
\end{yl}

To proceed further, we need  two expansion formulas of polynomials. The first one is Lemma 3.2 of \cite{wangjinpaper}. We record its proof below for completeness.
\begin{yl}\label{polybasis} Let $\{x_n\}_{n\geq 0}$ and $\{b_n\}_{n\geq 0}$ be two sequences such that $x_k\neq b_i, b_i\neq b_j, 0\leq i\neq j, k\leq N$. Then for any polynomial $f(x)$ of degree at most $N$, there holds
\begin{align}
f(x)=\sum_{k=0}^{N}\lambda_k\prod_{i=0}^{k-1}(b_i-x)\prod_{i=k+1}^{N}(x_i-x),
\label{assume-212}
\end{align}
where, for $n\leq N$,
\begin{align}
\lambda_n=(x_n-b_n)\sum_{k=0}^{n}
f(b_k)\prod_{i=n}^{N}\frac{1} {x_i-b_k}\prod_{i=0,i\neq k}^{n}\frac{1}{b_i-b_k}.\label{wangcoeff}
\end{align}
\end{yl}
\pf It suffices to show that any polynomial $f(x)$ in $x$ of degree at most $N$ can be expressed as a linear combination of polynomials
$$
\left\{\prod_{i=0}^{k-1}(b_i-x)\prod_{i=k+1}^{N}(x_i-x)\right\}_{k=0}^N.
$$
To prove this, let us assume that
\begin{align}
f(x)=\sum_{k=0}^{N}\lambda_k\prod_{i=0}^{k-1}(b_i-x)\prod_{i=k+1}^{N}(x_i-x).\label{assume}
\end{align}
It remains to determine the coefficients $\lambda_k$. For this, we observe that \eqref{assume} corresponds to the special case of  Lemma \ref{fglemma}, in which $f(x,y)=g(x,y)=x-y$  and
\begin{align*}
F(x)=\frac{f(x)}{\prod_{i=1}^{N}(x_i-x)}.
\end{align*}
Solving  \eqref{assume} for $\lambda_k$ by Lemma \ref{fglemma}, we obtain
\begin{align*}
\lambda_n&=(x_n-b_n)\sum_{k=0}^{n}F(b_k)\frac{\prod_{i=1}^{n-1} (x_i-b_k)}
 {\prod_{i=0,i\neq k}^{n}(b_i-b_k)}\\
 &=(x_n-b_n)\sum_{k=0}^{n}\frac{f(b_k)}
 {\prod_{i=n}^{N} (x_i-b_k)\prod_{i=0,i\neq k}^{n}(b_i-b_k)}.
\end{align*}
The lemma is confirmed.\qed

The above conclusion suggests that the set
$$
\bigg\{\prod_{i=0}^{k-1}(b_i-x)\prod_{i=k+1}^{N}(x_i-x)\bigg|x_i\neq b_j,i\neq j, 0\leq k\leq N\bigg\}
$$
is  a basis of vector space of polynomials of degree at most $N$.
From this point forward, we proceed to find another  basis for this vector space.

\begin{yl}\label{polybasistwo} Let $N\geq 0$ be   integer and $\{x_n\}_{n\geq 0}$ and $\{b_n\}_{n\geq 0}$ be two sequences such that $x_k\neq b_i, b_i\neq b_j, 0\leq i\neq j, k\leq N$. Then for any polynomial $f(x)$ of degree at most $N$, there holds
\begin{align}
f(x)=\sum_{k=0}^{N}\lambda_k\prod_{i=0}^{k-1}(b_i-x)\prod_{i=1}^{N-k}(x_i-x),\label{assume-214}
\end{align}
where, for $n\leq N$,
\begin{align}
\lambda_n=\sum_{k=0}^{n}f(b_k)
\frac{x_{N-n+1}-b_n} {x_{N-n+1}-b_k}\prod_{i=1}^{N-n}\frac{1} {x_i-b_k}\prod_{i=0,i\neq k}^{n}\frac{1}{b_i-b_k}.\label{secondcoefficient}
\end{align}
\end{yl}
\pf
 As previously, it suffices to show that any polynomial $f(x)$ in $x$ of degree at most $N$ can be expressed as a linear combination of polynomials
$$
\left\{\prod_{i=0}^{k-1}(b_i-x)\prod_{i=1}^{N-k}(x_i-x)\right\}_{k=0}^N.
$$
In other words, assume that
\begin{align}
f(x)=\sum_{k=0}^{N}\lambda_k\prod_{i=0}^{k-1}(b_i-x)\prod_{i=1}^{N-k}(x_i-x).\label{assume-new}
\end{align}
All we need to do is to find the coefficients $\lambda_k$. Unlike the proof for Lemma \ref{polybasis},  we proceed to find $\lambda_k$  via the use of recurrence relations. To that end, we take $x=b_n$ in  \eqref{assume-new} for $n=0,1,2,\ldots,N$ in succession. Then we have a system of linear equations in $N+1$ unknown $\lambda_k$'s as below:
\begin{align}
f(b_n)=\sum_{k=0}^{n}\lambda_k\prod_{i=0}^{k-1}(b_i-b_n)\prod_{i=1}^{N-k}(x_i-b_n).\label{assume-new-1}
\end{align}
Or equivalently, in terms of linear algebras,
\begin{align}
A X=\beta,\label{assume-new-2}
\end{align}
where the $(n,k)$-entry of the coefficient matrix $A$
$$
A_{n,k}=\prod_{i=0}^{k-1}(b_i-b_n)\prod_{i=1}^{N-k}(x_i-b_n),
$$
$X$ and $\beta$ stand, respectively, for the $(N+1)$-dimensional column vectors such that
\begin{align*}
X^T:=(\lambda_0,\lambda_1,\ldots,\lambda_N),\quad \beta^T:=(f(b_0),f(b_1),\ldots,f(b_N)),
\end{align*}
where the superscript $T$ denotes the transpose of vectors.
As it is easily seen, the coefficient matrix $A$ is lower-triangular with the $(n,n)$-entry
$$
A_{n,n}=\prod_{i=0}^{n-1}(b_i-b_n)\prod_{i=1}^{N-n}(x_i-b_n)\neq 0.
$$
Hence the linear equations \eqref{assume-new-2} has the unique solution. By solving  \eqref{assume-new-2} for $\lambda_k$, we obtain
\begin{align*}
X&=A^{-1}\beta,
\end{align*}
viz., for $A^{-1}=(B_{n,k})$,
\begin{align}
\lambda_n=\sum_{k=0}^n B_{n,k}f(b_k).\label{secondtemp}
\end{align}
It remains to find any explicit expression for $B_{n,k}$. To do this,  we first set up  certain  recurrence relation for the entries of $(B_{n,k})$.  What we obtained is  that for any $n-k\geq 1$,
\begin{align}
B_{n,k}=-\sum_{i=k}^{n-1}B_{i,k}\prod_{j=i}^{n-1}\frac{x_{N-j}-b_n}{b_j-b_n}
\label{inverse}
\end{align}
with the initial condition
$$
B_{k,k}=\prod_{i=0}^{k-1}\frac{1}{b_i-b_k}\prod_{i=1}^{N-k}\frac{1}{x_i-b_k}.
$$
As a matter of fact, \eqref{inverse} follows directly from the definition of inverse matrices \eqref{inversedef} as  below:
$$
(A_{n,k})(B_{n,k})=E_{N+1},
$$
where $E_{N+1}$ is  the identity matrix of order $N+1$.
Written out in explicit terms, it means that for $n-k\geq 1$,
\begin{align*}
A_{n,k}B_{k,k}+A_{n,k+1}B_{k+1,k}+\cdots+A_{n,n}B_{n,k}=0.
\end{align*}
Based on this relation, it is not hard to find
\begin{align*}
B_{n,k}=-\frac{A_{n,k}}{A_{n,n}}B_{k,k}-\frac{A_{n,k+1}}{A_{n,n}}B_{k+1,k}
-\cdots-\frac{A_{n,n-1}}{A_{n,n}}B_{n-1,k}.
\end{align*}
Finally, by solving the recurrence relation \eqref{inverse} with the initial condition $B_{k,k}$ by induction on $n-k\geq 0$, we have that
\begin{align}
B_{n,k}=\frac{x_{N-n+1}-b_n} {x_{N-n+1}-b_k}\prod_{i=1}^{N-n}\frac{1} {x_i-b_k}\prod_{i=0,i\neq k}^{n}\frac{1}{b_i-b_k}.\label{280-new}
\end{align}
Thus \eqref{secondcoefficient} follows after a direct substitution of \eqref{280-new} into  \eqref {secondtemp}. The lemma is proved.
\qed

Another ingredient we will use is  a basic fact about the symmetric-difference decomposition for the product of two theta functions. It is the foundation stone for our forthcoming discussions and  has already been proved  in \cite{wangjinpaper} by the first author.

\begin{yl}[{\rm Cf. \cite[Lemma 3.1]{wangjinpaper}}]\label{lemma2} Let $P(x)$ and $Q(x)$ be given by \eqref{PQdef}. Then
\begin{align}
y \theta(x y,x/y;p)=P(x)Q(y)-P(y)Q(x).\label{thetaid}
\end{align}
\end{yl}

At the end of this section, we  list some basic relations for the  elliptic $q,p$-shifted factorials which will be used later.

\begin{yl}[{\rm Cf. \cite[p. 301]{10}}] The following properties hold for theta functions and elliptic $q,p$-shifted factorials.
\begin{enumerate}[(i)]
\item $\theta(x;p)=-x\theta(1/x;p)$,~ $\theta(px;p)=-\displaystyle \frac{1}{x}\,\theta(x;p).$
\item The elliptic binomial coefficient
\begin{align}
\left[\begin{matrix}N\\ k\end{matrix}\right]_{q,p}:=\frac{(q;q,p)_N}{(q;q,p)_k(q;q,p)_{N-k}}
=\frac{q^{kN}}{\tau_q(k)}\frac{(q^{-N};q,p)_k}{(q;q,p)_k}.
\label{added-3}
\end{align}
\item\begin{align}&(ACq^k;q,p)_N=(AC;q,p)_{N}\frac{(ACq^N;q,p)_k}{(AC;q,p)_k},\label{added-1-1}\\
&(Aq^{-k}/C;q,p)_N=(A/C;q,p)_{N}q^{-kN}\frac{(Cq/A;q,p)_k}{(Cq^{-N+1}/A;q,p)_k}.\label{added-2}
\end{align}
\item \begin{align}
\theta(Cxq^{k},Cq^{k}/x;p)=\frac{(Cx,C/x;q,p)_{k+1}}{(Cx,C/x;q,p)_k}.\label{added-1}
\end{align}
\end{enumerate}
\end{yl}
We will also need Frenkel and Turaev's summation formula which is one of the most fundamental results for
VWP-balanced elliptic  hypergeometric series.

\begin{yl} [Frenkel and Turaev's summation formula. Cf. \cite{frenkel} or \mbox{\rm \cite[Eq. (11.2.25)]{10}}]\label{lemma28} For any integer $n\geq 0$ and complex numbers $a,b,c,d,e$ with $a^2q^{n+1}=bcde$, there holds
\begin{align}
{}_{10}V_9(a;b,c,d,e,q^{-n};q,p)=\frac{(aq,aq/(bc),aq/(bd),aq/(cd);q,p)_n}{(aq/b,aq/c,aq/d,aq/(bcd);q,p)_n}. \label{frenkel-sum}
\end{align}
\end{yl}
\section{The proofs  of the main theorems}
\subsection{The proofs of Theorems  \ref{mainthmchenfu} and \ref{mainthmyl-latest}}
As planned, we proceed to prove Theorem \ref{mainthmchenfu} first.

\begin{proof}[The proof of Theorem \ref{mainthmchenfu}]  We proceed as follows. At first, in light of Lemma \ref{polybasistwo}, we  only  need  to consider
\begin{align*}
F(x)=\sum_{k=0}^{N}\lambda_kx^k
\end{align*}
and thereby have the expansion
\begin{align}F(x)=\sum_{k=0}^{N}\Lambda_k(N)
\prod_{i=0}^{k-1}(b_i-x)\prod_{i=1}^{N-k}(x_i-x),\label{28-8-888333-new}
\end{align}
  where $\Lambda_k(N)$ is uniquely given by \eqref{secondcoefficient}. Next,  with the same $P(x)$ and $Q(x)$  given by \eqref{PQdef}, we  make the replacement or transformation of the parameters $x_i,b_i$ and  variable $x$
\begin{align}\label{transformation-new}
\left\{
  \begin{array}{rcl}
    x_i&\to& P(x_i)/Q(x_i)\\
    b_i&\to& P(b_i)/Q(b_i)\\ x&\to& P(x)/Q(x),
  \end{array}
\right.
\end{align}
for   \eqref{28-8-888333-new}, thereby obtaining
\begin{align*}
&F\big(P(x)/Q(x)\big)\\
&=
\sum_{k=0}^{N}\Lambda_k(N)
\prod_{i=0}^{k-1}\frac{P(b_i)Q(x)-Q(b_i)P(x)}{Q(b_i)Q(x)} \prod_{i=1}^{N-k}\frac{P(x_i)Q(x)-Q(x_i)P(x)}{Q(x_i)Q(x)}\\
&=\frac{1}{Q(x)^N}\sum_{k=0}^{N}\Lambda_k(N)\prod_{i=0}^{k-1}\frac{x\theta(b_ix,b_i/x;p)}{Q(b_i)}\prod_{i=1}^{N-k}\frac{x\theta(x_ix,x_i/x;p)}{Q(x_i)}.
\end{align*}
Note that the last equality results from Lemma \ref{lemma2}.
By considering  $$f(x):=Q(x)^NF(P(x)/Q(x))\in \mathcal{L}_N(P(x),Q(x)),$$ we obtain
\begin{align}
f(x)=x^N\sum_{k=0}^{N}\Lambda_k(N)
\prod_{i=0}^{k-1}\frac{\theta(b_ix,b_i/x;p)}{Q(b_i)}
\prod_{i=1}^{N-k}\frac{\theta(x_ix,x_i/x;p)}{Q(x_i)}.\label{RESULKT-new}
\end{align}
Note that, under the same transformation \eqref{transformation-new}, the coefficient $\Lambda_n(N)$ given by \eqref{secondcoefficient} takes the form
\begin{align*}
\Lambda_n(N)
&=\sum_{k=0}^{n}\frac{b_nQ(b_k)\theta(x_{N-n+1}b_n,x_{N-n+1}/b_n;p)}
{b_kQ(b_n)\theta(x_{N-n+1}b_k,x_{N-n+1}/b_k;p)}F\bigg(\frac{P(b_k)}{Q(b_k)}\bigg)\\
&\qquad\qquad\times
\prod_{i=1}^{N-n}\frac{Q(x_i)Q(b_k)}{b_k\theta(x_ib_k,x_i/b_k;p)}\prod_{i=0,i\neq k}^{n}
\frac{Q(b_i)Q(b_k)} {b_k\theta(b_ib_k,b_i/b_k;p)}\\
&=\frac{b_n}{Q(b_n)}\sum_{k=0}^{n}\frac{Q(b_k)}{b_k^{N+1}}\frac{\theta(x_{N-n+1}b_n,x_{N-n+1}/b_n;p)}
{\theta(x_{N-n+1}b_k,x_{N-n+1}/b_k;p)}f(b_k)\\
&\qquad\qquad\times\prod_{i=1}^{N-n}\frac{Q(x_i)}{\theta(x_ib_k,x_i/b_k;p)}\prod_{i=0,i\neq k}^{n}
\frac{Q(b_i)} {\theta(b_ib_k,b_i/b_k;p)}.
\end{align*} In the sequel, we redefine and then evaluate
\begin{align*}
H_n(N)&:=\frac{Q(b_n)\Lambda_n(N)}{b_n\prod_{i=0}^{n}Q(b_i)\prod_{i=1}^{N-n}Q(x_i)}\\
&=\sum_{k=0}^{n}\frac{1}{b_k^{N+1}}\frac{\theta(x_{N-n+1}b_n,x_{N-n+1}/b_n;p)}
{\theta(x_{N-n+1}b_k,x_{N-n+1}/b_k;p)}f(b_k)\prod_{i=1}^{N-n}\frac{1}{\theta(x_ib_k,x_i/b_k;p)}\prod_{i=0,i\neq k}^{n}
\frac{1} {\theta(b_ib_k,b_i/b_k;p)}.
\end{align*}
We have established \eqref{coeffi-chenfu-new}. As a last step, by substituting $H_n(N)$ for $\Lambda_n(N)$ in \eqref{RESULKT-new}, we readily find that
\begin{align*}
f(x)&=x^N\sum_{k=0}^{N}H_k(N)\bigg(b_k\prod_{i=0}^{k-1}Q(b_i)\prod_{i=1}^{N-k}Q(x_i)\bigg)
\prod_{i=0}^{k-1}\frac{\theta(b_ix,b_i/x;p)}{Q(b_i)}\prod_{i=1}^{N-k}\frac{\theta(x_ix,x_i/x;p)}{Q(x_i)}\\
&=x^N\sum_{k=0}^{N}H_k(N)b_k
\prod_{i=0}^{k-1}\theta(b_ix,b_i/x;p)\prod_{i=1}^{N-k}\theta(x_ix,x_i/x;p).
\end{align*}
Hence, we have proven the theorem.
\end{proof}

\begin{remark}
As far as we are aware,  the transformation \eqref{transformation-new} is of value since it may serve as a bridge connecting theta or elliptic function identities  such as \eqref{weierstrass} and polynomial identities like \eqref{tripleknown}. We refer the reader to \cite{wangjinpaper} for more  applications of \eqref{transformation-new}.
\end{remark}

\begin{remark} Analyzing the proofs of Theorems \ref{mainthmyl} and \ref{mainthmchenfu}, we think that any base of the ring
 of polynomials can be used to elliptic interpolation, provided that the expansion coefficients can be found easily via the technique of matrix inversions while the transformation \eqref{transformation-new} is available.
\end{remark}

In regard to  Theorem \ref{mainthmyl}, we remark that in \cite{wangjinpaper},  the author established Theorem \ref{mainthmyl} with the help of  Lemma \ref{polybasis}  and Transformation \eqref{transformation-new} but
 under a prior condition that two sequences $x_i\neq b_j$ for all $1\leq i,j\leq N$. This restriction arises from \eqref{wangcoeff} of Lemma \ref{polybasis}. As a matter of fact, if $x_i=b_j$ for all $1\leq i=j\leq N$, then we have a interpolation formula, viz., Theorem \ref{mainthmyl-latest}. The  proof is given as follows.

\begin{proof}[The proof of Theorem \ref{mainthmyl-latest}] It suffices to note that in \eqref{wangcoeff} of Lemma \ref{polybasis},  when $x_n=b_n$, the limitation
\begin{align*}
\lambda_n&=\lim_{x_n\to b_n}\bigg\{(x_n-b_n)\sum_{k=0}^{n}
f(b_k)\prod_{i=n}^{N}\frac{1} {x_i-b_k}\prod_{i=0,i\neq k}^{n}\frac{1}{b_i-b_k}\bigg\}=
f(b_n)\prod_{i=0,i\neq n}^{N}\frac{1}{b_i-b_n}.
\end{align*}
As such, we rediscover the Lagrange interpolation formula
\begin{align}
f(x)=\sum_{k=0}^{N}f(b_k)\prod_{i=0,i\neq k}^{N}\frac{b_i-x}{b_i-b_k}.\label{assume-new2}
\end{align}
As been done for Theorem  \ref{mainthmchenfu}, we  apply Transformation  \eqref{transformation-new} to both sides of             \eqref{assume-new2},
obtaining
\begin{align*}
&f\big(P(x)/Q(x)\big)\\
&=
\sum_{k=0}^{N}f(P(b_k)/Q(b_k))
\prod_{i=0,i\neq k}^{N}\frac{P(b_i)Q(x)-Q(b_i)P(x)}{Q(b_i)Q(x)} \prod_{i=0,i\neq k}^{N}\frac{Q(b_i)Q(b_k)}{P(b_i)Q(b_k)-Q(b_i)P(b_k)}\\
&=\frac{x^N}{Q(x)^N}\sum_{k=0}^{N}\frac{Q(b_k)^Nf(P(b_k)/Q(b_k))}{b_k^N}\prod_{i=0,i\neq k}^{N}\frac{\theta(b_ix,b_i/x;p)}{\theta(b_ib_k,b_i/b_k)}.
\end{align*}
By abbreviating $Q(x)^Nf\big(P(x)/Q(x)\big)\in \mathcal{L}_N(P(x),Q(x)) $ with $f(x)$, we obtain \eqref{13-13} at once.
\end{proof}

When the base $p=0$ in \eqref{13-13},  we achieve at once an interesting interpolation formula for rational functions at points $b_0,b_1,\ldots,b_N$.
\begin{lz}
For nonnegative integers  $m\leq N$, there holds
\begin{align*}
\left(\frac{1+x^2}{x}\right)^m=\sum_{k=0}^{N}\left(\frac{1+b_k^2}{b_k}\right)^m \prod_{i=0,i\neq k}^{N}
\frac{(1-b_ix)(1-b_i/x)}
{(1-b_ib_k)(1-b_i/b_k)}.
\end{align*}
\end{lz}
\subsection{Two proofs for  Theorem \ref{mainthm}}

Now we are in a good position to  show Theorem \ref{mainthm}  in full details.
As we will see below, it is  implied by   both Theorem  \ref{mainthmyl} and Theorem \ref{mainthmchenfu}.

\begin{proof}[The first proof of Theorem \ref{mainthm}]  It is a direct application of Theorem \ref{mainthmchenfu}. To make this clear, we first  exchange the   order
of summations in \eqref{mathdlchenfu} to get
\begin{align}f(x)=x^N\sum_{t=0}^{N}\frac{f(b_t)}{b_t^{N+1}}\prod_{i=0}^{t-1}\frac{\theta(b_ix,b_i/x;p)}{\theta(b_ib_t,b_i/b_t;p)} \mathbf{S}_{N,t},\label{finale}
\end{align}
where
\begin{align*}
\mathbf{S}_{N,t}:=\sum_{k=t}^{N}b_k\frac{\theta(x_{N-k+1}b_k,x_{N-k+1}/b_k;p)}
{\theta(x_{N-k+1}b_t,x_{N-k+1}/b_t;p)}\prod_{i=1}^{N-k}\frac{\theta(x_ix,x_i/x;p)}{\theta(x_ib_t,x_i/b_t;p)}
\prod_{i=t+1}^{k}\frac{\theta(b_{i-1}x,b_{i-1}/x;p)} {\theta(b_ib_t,b_i/b_t;p)}.
\end{align*} Next we choose $$x_i=Aq^{-i},  b_i=Cq^i.$$ In this case, it is easy to check that
\begin{align*}
&\prod_{i=1}^{N-k}\frac{\theta(x_ix,x_i/x;p)}{\theta(x_ib_t,x_i/b_t;p)}=\frac{(xAq^{k-N},Aq^{k-N}/x;q,p)_{N-k}}
{(ACq^{t+k-N},Aq^{-t+k-N}/C;q,p)_{N-k}},\\
&\prod_{i=t+1}^{k}\frac{\theta(b_{i-1}x,b_{i-1}/x;p)} {\theta(b_ib_t,b_i/b_t;p)}=\frac{(xCq^t,Cq^{t}/x;q,p)_{k-t}}{(q,C^2q^{2t+1};q,p)_{k-t}}.
\end{align*}
Therefore, upon changing the index $k$ to $K$ by the relation $K=k-t$, we are able to compute
\begin{align*}
\mathbf{S}_{N,t}&=Cq^{t}\sum_{K=0}^{N-t}q^{K}\frac{\theta(ACq^{2K+2t-N-1},Aq^{-N-1}/C;p)}
{\theta(ACq^{K+2t-N-1},Aq^{K-N-1}/C;p)}\\
&\times\frac{(xAq^{K+t-N},Aq^{K+t-N}/x;q,p)_{N-t-K}}
{(ACq^{K+2t-N},Aq^{K-N}/C;q,p)_{N-t-K}}\,
\frac{(xCq^t,Cq^{t}/x;q,p)_{K}}{(q,C^2q^{2t+1};q,p)_{K}}.
\end{align*}
Since
\begin{align*}
&\frac{(xAq^{K+t-N},Aq^{K+t-N}/x;q,p)_{N-t-K}}
{(ACq^{K+2t-N},Aq^{K-N}/C;q,p)_{N-t-K}}=\frac{(xAq^{t-N},Aq^{t-N}/x;q,p)_{N-t}}
{(ACq^{2t-N},Aq^{-N}/C;q,p)_{N-t}}\times\frac{(ACq^{2t-N},Aq^{-N}/C;q,p)_{K}}{(xAq^{t-N},Aq^{t-N}/x;q,p)_{K}},
\end{align*}
it is easily found that
\begin{multline*}
 \mathbf{S}_{N,t}=Cq^{t}\frac{(xAq^{t-N},Aq^{t-N}/x;q,p)_{N-t}}
{(ACq^{2t-N},Aq^{-N}/C;q,p)_{N-t}}\\
\times\sum_{K=0}^{N-t}q^{K}\frac{\theta(ACq^{2K+2t-N-1};p)}
{\theta(ACq^{2t-N-1};p)}
\frac{\theta(ACq^{2t-N-1},Aq^{-N-1}/C;p)}
{\theta(ACq^{K+2t-N-1},Aq^{K-N-1}/C;p)}\\
\times\frac{(ACq^{2t-N},Aq^{-N}/C;q,p)_{K}}{(q,C^2q^{2t+1};q,p)_{K}}\,
\frac{(xCq^t,Cq^{t}/x;q,p)_{K}}{(Aq^{t-N}/x,xAq^{t-N};q,p)_{K}}.
\end{multline*}
Using the relation
\begin{align*}
\frac{\theta(ACq^{2t-N-1},Aq^{-N-1}/C;p)}
{\theta(ACq^{K+2t-N-1},Aq^{K-N-1}/C;p)}=\frac{(ACq^{2t-N-1},Aq^{-N-1}/C;q,p)_K}
{(ACq^{2t-N},Aq^{-N}/C;q,p)_K}
\end{align*}
and then recasting the last sum in  standard notation of elliptic hypergeometric series,  we arrive at
\begin{align*}
\mathbf{S}_{N,t}&=Cq^{t}\frac{(xAq^{t-N},Aq^{t-N}/x;q,p)_{N-t}}
{(ACq^{2t-N},Aq^{-N}/C;q,p)_{N-t}}\\
&\times{}_{10}V_9\big(ACq^{2t-N-1};Aq^{-N-1}/C,Cq^tx,Cq^t/x, ACq^{t},q^{-(N-t)};q,p\big).
\end{align*}
Now, by appealing to Frenkel and Turaev's summation formula \eqref{frenkel-sum},
 we readily find that
 \begin{align*}
\mathbf{S}_{N,t}&=Cq^{t}\frac{(xAq^{t-N},Aq^{t-N}/x;q,p)_{N-t}}
{(ACq^{2t-N},Aq^{-N}/C;q,p)_{N-t}}\times\frac{(ACq^{2t-N},Cxq^{t+1},Cq^{t+1}/x,Aq^{-N}/C;q,p)_{N-t}}
{(C^2q^{2t+1},Aq^{t-N}/x,Axq^{t-N},q;q,p)_{N-t}}
\\
&=Cq^{t}\frac{(Cxq^{t+1},Cq^{t+1}/x;q,p)_{N-t}}
{(q,C^2q^{2t+1};q,p)_{N-t}}.
\end{align*}
A substitution of this expression simplifies the preceding expansion \eqref{finale} to
\begin{align}f(x)&=x^N\sum_{t=0}^{N}\frac{f(Cq^{t})}{(Cq^{t})^{N}}\frac{(Cx,C/x;q,p)_t}{(C^2q^t,q^{-t};q,p)_t}\frac{(Cxq^{t+1},Cq^{t+1}/x;q,p)_{N-t}}
{(q,C^2q^{2t+1};q,p)_{N-t}}\nonumber\\
&=\bigg(\frac{x}{C}\bigg)^N\frac{(Cx,C/x;q,p)_{N+1}}{(q;q,p)_{N}}\sum_{t=0}^{N}\left[\begin{matrix}N\\ t\end{matrix}\right]_{q,p}\tau_q(t)q^{t-tN}\frac{\theta(C^2q^{2t};p)f(Cq^t)}
{(C^2q^{t};q,p)_{N+1}\theta(Cxq^{t},Cq^{t}/x;p)}\nonumber\\
&=\bigg(\frac{x}{C}\bigg)^N\frac{(Cx,C/x;q,p)_{N+1}}{(q,C^2q;q,p)_{N}}
\sum_{t=0}^{N}\frac{\theta(C^2q^{2t};p)}
{\theta(C^2;p)}\frac{(q^{-N},C^2;q,p)_{t}}
{(q,C^2q^{N+1};q,p)_{t}}\frac{f(Cq^t)q^{t}}{\theta(Cxq^{t},Cq^{t}/x;p)}.\label{better}
\end{align}
Again, referring to  \eqref{added-1}, it is clear that
$$
\frac{1}{\theta(Cxq^{t},Cq^{t}/x;p)}=\frac{1}{\theta(Cx,C/x;p)}\frac{(C/x,Cx;q,p)_t}{(Cxq,Cq/x;q,p)_t}.
$$
Therefore, we reduce  \eqref{better} to
\begin{align*}f(x)=\bigg(\frac{x}{C}\bigg)^N\frac{(Cxq,Cq/x;q,p)_{N}}{(q,C^2q;q,p)_{N}}
\sum_{t=0}^{N}\frac{\theta(C^2q^{2t};p)}
{\theta(C^2;p)}\frac{(C^2,C/x,Cx,q^{-N};q,p)_{t}}
{(q,Cxq,Cq/x,C^2q^{N+1};q,p)_{t}}f(Cq^t)q^{t}.
\end{align*}
This is precisely what we want.  The theorem is proved.
\end{proof}

The next derivation for  Theorem \ref{mainthm} is based on Theorem \ref{mainthmyl}.

\begin{proof}[The second proof of Theorem \ref{mainthm}]  It only needs to specialize Theorem \ref{mainthmyl}  to the case  $$x_i=Bq^{i-1}, b_i=Cq^{i}.$$
Consequently,  the expansion \eqref{13} reduces to
\begin{align*}
\frac{f(x)}{(Bx,B/x;q,p)_N}
=x^N\sum_{k=0}^{N}H_k(N)
\frac{(Cx,C/x;q,p)_{k}}{(Bx,B/x;q,p)_k},
\end{align*}
where the coefficient $H_k(N)$ is given by \eqref{coeffi}, viz.,
\begin{align*}
&H_k(N)=Cq^{k}~\theta(BCq^{2k-1},B/(Cq);p)\\
&\times\sum_{t=0}^{k}\frac{1}{(Cq^t)^{N+1}}
\frac{f(Cq^t)\theta(C^2q^{2t};p)} {(BCq^{t+k-1},Bq^{k-1-t}/C;q,p)_{N-k+1}(C^2q^t;q,p)_{k+1}
}\prod_{i=0,i\neq t}^{k}\frac{1}{\theta(q^{i-t};p)}.
\end{align*}
Note that
$$
\prod_{i=0,i\neq t}^{k}\frac{1}{\theta(q^{i-t};p)}=\frac{\tau_q(t)q^t}{(q;q,p)_t(q;q,p)_{k-t}}.
$$
Therefore,
\begin{align*}
H_k(N)&=\frac{q^{k}}{C^N}~\theta(BCq^{2k-1},B/(Cq);p)\\
&\times\sum_{t=0}^{k}
\frac{f(Cq^t)\theta(C^2q^{2t};p)} {(BCq^{t+k-1},Bq^{k-1-t}/C;q,p)_{N-k+1}(C^2q^t;q,p)_{k+1}}\frac{\tau_q(t)q^{-Nt}}{(q;q,p)_t(q;q,p)_{k-t}}.
\end{align*}
Substituting the above and exchanging the order of summations gives rise to
 \begin{align}
\frac{f(x)}{(Bx,B/x;q,p)_N}=\bigg(\frac{x}{C}\bigg)^N\theta(B/(Cq);p)
\sum_{t=0}^{N}\tau_q(t)q^{-Nt}\frac{\theta(C^2q^{2t};p)}{(q;q,p)_t}f(Cq^t)~\mathbf{T}_{N,t},\label{sumtemp}
\end{align}
where, we write $\mathbf{T}_{N,t}$ for the corresponding inner sum, i.e.,
\begin{align*}
\mathbf{T}_{N,t}&:=
\sum_{k=t}^{N}
\frac{\theta(BCq^{2k-1};p)} {(BCq^{t+k-1},Bq^{k-1-t}/C;q,p)_{N-k+1}(C^2q^t;q,p)_{k+1}
}\\
&\qquad\times \frac{(Cx,C/x;q,p)_{k}}{(Bx,B/x;q,p)_k}\frac{q^k}{(q;q,p)_{k-t}}.
\end{align*}
Further, by changing the index $k$ of summation to $K$ by the relation $K=k-t$, we arrive at
\begin{align*}
\mathbf{T}_{N,t}&=\frac{\theta(BCq^{2t-1};p)} {(BCq^{2t-1},Bq^{-1}/C;q,p)_{N-t+1}(C^2q^t;q,p)_{t+1}}\frac{(Cx,C/x;q,p)_{t}}{(Bx,B/x;q,p)_t}q^t\\
&\times\sum_{K=0}^{N-t}\frac{\theta(BCq^{2K+2t-1};p)}{\theta(BCq^{2t-1};p)}
\frac{(BCq^{2t-1},Bq^{-1}/C;q,p)_{K}} {(q,C^2q^{2t+1};q,p)_{K}}
\frac{(Cxq^{t},Cq^{t}/x;q,p)_{K}}{(Bxq^{t},Bq^{t}/x;q,p)_{K}}q^{K}
\\
&=\frac{\theta(BCq^{2t-1};p)} {(BCq^{2t-1},Bq^{-1}/C;q,p)_{N-t+1}(C^2q^t;q,p)_{t+1}}\frac{(Cx,C/x;q,p)_{t}}{(Bx,B/x;q,p)_t}q^t\\
&\qquad\qquad\times{}_{10}V_9\big(BCq^{2t-1};Bq^{-1}/C,Cxq^t,Cq^t/x, BCq^{t+N},q^{-(N-t)};q,p\big).
\end{align*}
Now, by Frenkel and Turaev's summation formula \eqref{frenkel-sum},
 we can evaluate $\mathbf{T}_{N,t}$ in closed form:
\begin{align*}
\mathbf{T}_{N,t}
&=\frac{\theta(BCq^{2t-1};p)} {(BCq^{2t-1},Bq^{-1}/C;q,p)_{N-t+1}(C^2q^t;q,p)_{t+1}}\frac{(Cx,C/x;q,p)_{t}}{(Bx,B/x;q,p)_t}q^t\\
&\qquad\qquad\times\frac{(B C q^{2 t},C q^{t+1}/x,C x q^{t+1},B/C;q,p)_{N-t}}{(C^2 q^{2 t+1},B q^t/x, B x q^t,q;q,p)_{N-t}}\\
&=\frac{(Cx,C/x;q,p)_{N+1}}{\theta(Bq^{-1}/C;p)(C^2;q,p)_{N+1}(Bx,B/x;q,p)_N}\\
&\qquad\qquad\times\frac{q^t}{(q;q,p)_{N-t}}\frac{(C^2;q,p)_{t}} {\theta(Cxq^t,Cq^t/x;p)(C^2q^{N+1};q,p)_{t}}.
\end{align*}
In the sequel, on substituting this computational result into \eqref{sumtemp}, we have the expansion
\begin{align*}
\frac{f(x)}{(Bx,B/x;q,p)_N}&=\bigg(\frac{x}{C}\bigg)^N\frac{(Cx,C/x;q,p)_{N+1}}{(C^2;q,p)_{N+1}(Bx,B/x;q,p)_N}\\
&\times\sum_{t=0}^{N}\tau_q(t)q^{t-Nt}\frac{f(Cq^t)}{(q;q,p)_t(q;q,p)_{N-t}}
\frac{\theta(C^2q^{2t};p)(C^2;q,p)_{t}} {\theta(Cxq^t,Cq^t/x;p)(C^2q^{N+1};q,p)_{t}}\\
&=\bigg(\frac{x}{C}\bigg)^N\frac{(Cx,C/x;q,p)_{N+1}}{(C^2;q,p)_{N+1}(q,Bx,B/x;q,p)_N}\\
&\times\sum_{t=0}^{N}\left[\begin{matrix}N\\ t\end{matrix}\right]_{q,p}\tau_q(t)q^{t-Nt}\frac{\theta(C^2q^{2t};p)(C^2;q,p)_{t}} {\theta(Cxq^t,Cq^t/x;p)(C^2q^{N+1};q,p)_{t}}f(Cq^t).
\end{align*}
At this stage, we utilize the properties \eqref{added-3} and \eqref{added-1} to simplify the last expansion, obtaining
\begin{align*}
f(x)&=\bigg(\frac{x}{C}\bigg)^N\frac{(Cxq,Cq/x;q,p)_{N}}{(q,C^2q;q,p)_{N}}\\
&\times\sum_{t=0}^{N}\frac{\theta(C^2q^{2t};p)} {\theta(C^2;p)}\frac{(C^2,C/x,Cx,q^{-N};q,p)_{t}} {(q,Cxq,Cq/x,C^2q^{N+1};q,p)_{t}}f(Cq^t)q^{t}.
\end{align*}
 It gives the complete proof of the theorem.
\end{proof}
\begin{remark}
It is worth pointing out that \eqref{xe33-3} is   independent of $A$ and $B$, although we derive \eqref{xe33-3}  by setting either $x_i=Aq^{-i}$ or $x_i=Bq^{i-1}$ .  We believe that other choices  for $x_i$ and $b_i$ also deserve further study, in order to find interpolation formulas in closed form.
\end{remark}
\subsection{The proof of Theorem \ref{generialized} and a characterization of $W_c^N$}
Indeed, a combination of Theorem  \ref{mainthm} with the basic relation \eqref{pqdecomp} or Lemma \ref{lemma2} leads us to a proof of Theorem \ref{generialized}.

\begin{proof}[The proof of Theorem \ref{generialized}] Given such $f(x)\in \mathcal{L}_{N_0}(P(x),Q(x))$, we only need to consider the theta function
\begin{align}
F(x)=f(x)\prod_{i=1}^mx^{N_i}(A_ix,A_i/x;q,p)_{N_i}.\label{Karlsson}
\end{align}
In this case, by Proposition \ref{closed}, we see that  $F(x)\in \mathcal{L}_{N}(P(x),Q(x))$ being of degree $N=N_0+N_1+\cdots+N_m$. Applying Theorem \ref{mainthm} to $F(x)$, we thereby have the expansion
\begin{align*}&f(x)\prod_{i=1}^mx^{N_i}(A_ix,A_i/x;q,p)_{N_i}=\bigg(\frac{x}{C}\bigg)^{N}\frac{(Cxq,Cq/x;q,p)_{N}}{(q,C^2q;q,p)_{N}}\nonumber\\
&\qquad\times\sum_{k=0}^{N}\frac{\theta(C^2q^{2k};p)} {\theta(C^2;p)}\frac{(C^2,C/x,Cx,q^{-N};q,p)_{k}} {(q,Cxq,Cq/x,C^2q^{N+1};q,p)_{k}}q^{k}\\
&\qquad\times f(Cq^k)\big(Cq^k\big)^{N-N_0}\prod_{i=1}^m(A_iCq^k,A_iq^{-k}/C;q,p)_{N_i}.
\end{align*}
By \eqref{added-1-1} and \eqref{added-2}, it is easy to check that
\begin{align*}
&\prod_{i=1}^m(A_iCq^k,A_iq^{-k}/C;q,p)_{N_i}\\
&=q^{-k(N-N_0)}\prod_{i=1}^m(A_iC,A_i/C;q,p)_{N_i}\prod_{i=1}^m
\frac{(A_iCq^{N_i},Cq/A_i;q,p)_k}{(Cq^{-N_i+1}/A_i,A_iC;q,p)_k}.
\end{align*}
In conclusion, we obtain
\begin{align*}&f(x)=\bigg(\frac{x}{C}\bigg)^{N_0}\frac{(Cxq,Cq/x;q,p)_{N}}{(q,C^2q;q,p)_{N}}\prod_{i=1}^m\frac{(A_iC,A_i/C;q,p)_{N_i}}{(A_ix,A_i/x;q,p)_{N_i}}\nonumber\\
&\times\sum_{k=0}^{N}\frac{\theta(C^2q^{2k};p)} {\theta(C^2;p)}\frac{(C^2,C/x,Cx,q^{-N};q,p)_{k}} {(q,Cxq,Cq/x,C^2q^{N+1};q,p)_{k}}q^{k}f(Cq^k)\prod_{i=1}^m\frac{(A_iCq^{N_i},Cq/A_i;q,p)_k}{(Cq^{-N_i+1}/A_i,A_iC;q,p)_k}.
\end{align*}
This completes the proof of the theorem.
\end{proof}

It seems interesting that we are able to characterize,  by making use of Schlosser and Yoo's relevant results and our argument,  the set $W_c^N$  in terms of   $P(x)$ and $Q(x)$.

\begin{proof}[The proof of Theorem \ref{chara}] We begin with proving \eqref{oneoneone}  first. For this, recall that for any integer $N\geq 0$, by Schlosser's conclusion \cite[Lemma 4.1]{schlosseradd}, we have the expansion
\begin{align}
\frac{g_N(x)}{(cx,c/x;q,p)_N}=\sum_{k=0}^{N}f_k\frac{(ax,a/x;q,p)_k}{(cx,c/x;q,p)_k}.\label{5.4}
\end{align}
It can be reformulated as
\begin{align}
g_N(x)=\sum_{k=0}^{N}f_k(ax,a/x;q,p)_k(cxq^k,cq^k/x;q,p)_{N-k}.\label{oneone}
\end{align}
Observe that
\begin{align*}
(ax,a/x;q,p)_k&=\prod_{i=0}^{k-1}\theta(aq^ix,aq^i/x;p)=\frac{1}{x^k}\prod_{i=0}^{k-1}\big(P(aq^i)Q(x)-Q(aq^i)P(x)\big),\\
(cxq^k,cq^k/x;q,p)_{N-k}&=\prod_{i=0}^{N-k-1}\theta(cq^{k+i}x,cq^{k+i}/x;p)\\
&=\frac{1}{x^{N-k}}\prod_{i=0}^{N-k-1}\big(P(cq^{k+i})Q(x)-Q(cq^{k+i})P(x)\big).
\end{align*}
Referring to Definition \ref{PQcondition} and Proposition \ref{closed}, it is easily seen that
$$
x^N(ax,a/x;q,p)_k(cxq^k,cq^k/x;q,p)_{N-k}
$$
is an elliptic Askey-Wilson polynomial of degree $N$. Assume further that
\begin{align}
x^N(ax,a/x;q,p)_k(cxq^k,cq^k/x;q,p)_{N-k}=\sum_{i=0}^{N}\mu_{N,k;i}P(x)^iQ(x)^{N-i},\label{newtwotwo}
\end{align}
where the coefficients  $\mu_{N,k;i}$ are independent of $x$.
Finally we substitute \eqref{newtwotwo} into \eqref{oneone},  arriving at
\begin{align*}
g_N(x)&=\frac{1}{x^N}\sum_{k=0}^{N}f_k\sum_{i=0}^{N}\mu_{N,k;i}P(x)^iQ(x)^{N-i}\\
&=\frac{1}{x^N}\sum_{i=0}^{N}\bigg(\sum_{k=0}^{N}f_k\mu_{N,k;i}\bigg)P(x)^iQ(x)^{N-i}.
\end{align*}
By choosing $\lambda_i=\sum_{k=0}^{N}f_k\mu_{N,k;i}$, we get  \eqref{oneoneone} of our theorem. For the case $m\leq N$, we only need to replace $g_N(x)$ of \eqref{oneoneone} with
$$
g_m(x)\frac{(cx,c/x;q,p)_N}{(cx,c/x;q,p)_m}=g_m(x)(cq^mx,cq^m/x;q,p)_{N-m}
$$
and apply Lemma \ref{lemma2} to the factor  $(cq^mx,cq^m/x;q,p)_{N-m}$.
Then \eqref{threethreethree} follows. This completes the proof of the theorem.\end{proof}

\begin{remark} By \eqref{oneoneone} and the definitions of  $P(x)$ and $Q(x)$, it is  easily verified that
$g_N(x)$  satisfies \eqref{2-2-22}, viz.,
\begin{align*}
\left\{
   \begin{array}{l}
     g_N(x)=g_N(1/x),\\
    g_N(px)=\displaystyle\frac{1}{p^Nx^{2N}}g_N(x).
   \end{array}
 \right.
\end{align*}
Likewise,   we readily find  that the elliptic Askey-Wilson polynomial $f(x)\in \mathcal{L}_N(P(x),Q(x))$ possesses the properties
\begin{align}
     f(1/x)=f(px)=\displaystyle\frac{1}{x^{2N}}f(x).
\end{align}
\end{remark}

\section{Applications}

In this section, we will pursue some specific cases of our interpolation  formulas, viz., Theorems \ref{mainthmyl-latest}, \ref{mainthm}, and \ref{generialized}.  We begin by establishing   a generalization of   Weierstrass' theta identity. It is  an application of   Theorem \ref{mainthmyl-latest}   with  Theorem \ref{mainthmyl} utilized.

\begin{tl}[Generalized Weierstrass theta identity]\label{dl4.7} Let $\{x_n\}_{n\geq 0}$ and $\{b_n\}_{n\geq 0}$ be two sequences such that $b_i\neq b_j, 0\leq i\neq j\leq N$. Then, for integers $N\geq k\geq 0$, there holds
\begin{align}
&\sum_{n=k}^{N}b_n\frac{\theta(x_nb_n,x_n/b_n;p)}{\theta(x_nb_k,x_n/b_k;p)}\prod_{i=n+1}^{N}
\frac{\theta(x_ix,x_i/x;p)}
{\theta(x_ib_k,x_i/b_k;p)}\label{generalizedweier}\\
&\qquad\quad\times\frac{\prod_{i=0}^{n-1}\theta(b_ix,b_i/x;p)}{\prod_{i=0,i\neq k}^{n}\theta(b_ib_k,b_i/b_k;p)}
=b_k\prod_{i=0,i\neq k}^{N}\frac{\theta(b_ix,b_i/x;p)}{\theta(b_ib_k,b_i/b_k;p)}.\nonumber
\end{align}
\end{tl}
\begin{proof} It suffices to consider two expansions for any $f(x)\in \mathcal{L}_N(P(x),Q(x))$. At first, according  to Theorem \ref{mainthmyl}, we have the expansion
\begin{align*}\frac{f(x)}{x^N}=
\sum_{n=0}^{N}H_n(N)b_n\theta(x_nb_n,x_n/b_n;p)
\prod_{i=0}^{n-1}\theta(b_ix,b_i/x;p)\prod_{i=n+1}^{N}\theta(x_ix,x_i/x;p),\nonumber
\end{align*}
where the coefficients
\begin{align*}
H_n(N)&:=\sum_{k=0}^{n}\frac{f(b_k)}{b_k^{N+1}}\prod_{i=n}^{N}\frac{1} {\theta(x_ib_k,x_i/b_k;p)}\prod_{i=0,i\neq k}^{n}\frac{1} {\theta(b_ib_k,b_i/b_k;p)}.
\end{align*}
By exchanging the order of summations, we have
\begin{align}
\frac{f(x)}{x^N}=\sum_{k=0}^{N}\frac{f(b_k)}{b_k^{N}}\Omega_{k,N},\label{expansion-one}
\end{align}
where
\begin{align*}
\Omega_{k,N}:=\frac{1}{b_k}\sum_{n=k}^{N}b_n\theta(x_nb_n,x_n/b_n;p)\frac{\prod_{i=0}^{n-1}\theta(b_ix,b_i/x;p)\prod_{i=n+1}^{N}\theta(x_ix,x_i/x;p)}{\prod_{i=n}^{N}\theta(x_ib_k,x_i/b_k;p)\prod_{i=0,i\neq k}^{n}\theta(b_ib_k,b_i/b_k;p)}.
\end{align*}
On the other hand, from Theorem \ref{mainthmyl-latest} it follows that
\begin{align}
\frac{f(x)}{x^N}=\sum_{k=0}^{N}\frac{f(b_k)}{b_k^N}\prod_{i=0,i\neq k}^{N}\frac{\theta(b_ix,b_i/x;p)}{\theta(b_ib_k,b_i/b_k;p)}.\label{expansion-two}
\end{align}
By comparing  \eqref{expansion-one} and \eqref{expansion-two} and taking the arbitrariness of $f(x)$ into account, we have
\begin{align}
\Omega_{k,N}=\prod_{i=0,i\neq k}^{N}\frac{\theta(b_ix,b_i/x;p)}{\theta(b_ib_k,b_i/b_k;p)}.\label{expansion-three}
\end{align}
Written out in full, \eqref{expansion-three}  is identified with  \eqref{generalizedweier}. The proof is finished.
\end{proof}
We should  make some remarks on the implications of Corollary \ref{dl4.7}.
\begin{remark}
 It is worth mentioning that Corollary \ref{dl4.7}  offers a general version of  Weierstrass' theta identity \eqref{weierstrass}.  To make this clear, we only need to specialize   \eqref{generalizedweier} to the case $N=2$ and $k=1$, obtaining
\begin{multline*}
b_1
\frac{\theta(x_2x,x_2/x;p)}
{\theta(x_2b_1,x_2/b_1;p)}\frac{\theta(b_0x,b_0/x;p)}
{\theta(b_0b_1,b_0/b_1;p)}\\+b_2\frac{\theta(x_2b_2,x_2/b_2;p)}{\theta(x_2b_1,x_2/b_1;p)}
\frac{\theta(b_0x,b_0/x;p)\theta(b_1x,b_1/x;p)}
{\theta(b_0b_1,b_0/b_1;p)\theta(b_2b_1,b_2/b_1;p)}\\
=b_1\frac{\theta(b_0x,b_0/x;p)}{\theta(b_0b_1,b_0/b_1;p)}\frac{\theta(b_2x,b_2/x;p)}{\theta(b_2b_1,b_2/b_1;p)}.
\end{multline*}
After some routine simplification, we have
\begin{align*}
\theta(b_2x,b_2/x,x_2b_1,x_2/b_1;p)-\theta(x_2x,x_2/x,b_2b_1,b_2/b_1;p)=\frac{b_2}{b_1} \theta(x_2b_2,x_2/b_2,b_1x,b_1/x;p).
\end{align*}
It is in agreement with \eqref{weierstrass} after relableling the parameters.
\end{remark}

The case $k=0$ of \eqref{generalizedweier} yields another theta identity being equivalent to \eqref{weierstrass}.
\begin{lz}\label{dl4.77}For any integer $N\geq  0$, there holds
\begin{align}
\sum_{n=0}^{N}b_n\frac{\theta(b_0x,b_0/x,x_nb_n,x_n/b_n;p)}{\theta(x_nb_0,x_n/b_0,b_nx,b_n/x;p)}\prod_{i=n+1}^{N}
\frac{\theta(b_{i}b_0,b_{i}/b_0,x_ix,x_i/x;p)}
{\theta(x_ib_0,x_i/b_0,b_{i}x,b_{i}/x;p)}
=b_0.\label{generalizedweier-333}\end{align}
\end{lz}

In regard to applications of Theorem \ref{mainthm},  two easy cases arise naturally.
\begin{lz} For any integer $N\geq 0$, we have
 \begin{align}
 &\bigg(\frac{C}{x}\bigg)^N\frac{(q,C^2q;q,p)_{N}}{(Cxq,Cq/x;q,p)_{N}}
 \theta^N(-x^2;p^2)\label{AAA}\\
&\qquad=\sum_{k=0}^{N}\frac{\theta(C^2q^{2k};p)} {\theta(C^2;p)}\frac{(C^2,Cx,C/x,q^{-N};q,p)_{k}} {(q,Cxq,Cq/x,C^2q^{N+1};q,p)_{k}}\theta^N(-C^2q^{2k};p^2)q^{k},
\nonumber\\
&\frac{(q,C^2q;q,p)_{N}}{(Cxq,Cq/x;q,p)_{N}}\theta^N(-px^2;p^2)
\label{BBB}\\
&\qquad=\sum_{k=0}^{N}\left[\begin{matrix}N\\ k\end{matrix}\right]_{q,p}\frac{\theta(C^2q^{2k};p)} {\theta(C^2;p)}\frac{(C^2,Cx,C/x;q,p)_{k}} {(Cxq,Cq/x,C^2q^{N+1};q,p)_{k}}\theta^N(-pC^2q^{2k};p^2)\tau_q(k)q^{k}.\nonumber
\end{align}
\end{lz}
\pf It suffices to specialize  \eqref{xe33-3} to the cases
$
f(x)=P(x)^{N}=\theta^N(-x^2;p^2)(-p;p)^N_\infty
$
 and $
f(x)=Q(x)^{N}
=(x\theta(-px^2;p^2)(-p;p)_\infty)^N.
$
Then we have  \eqref{AAA} and   \eqref{BBB} correspondingly.
\qed

As a very good illustration of Theorem  \ref{mainthm},  we prefer to reconsider Example 3.5 of \cite{wangjinpaper}. It gives a special result of Frenkel and Turaev's well-known summation formula \eqref{frenkel-sum} given by Lemma \ref{lemma28}.
\begin{lz} For any integer $N\geq 0$, there holds \begin{align}
{}_{10}V_9(C^2;C/x,Cx,Cq/A,ACq^N,q^{-N};q,p)=\frac{(q,C^2q,Ax,A/x;q,p)_N}{(AC,A/C,Cxq,Cq/x;q,p)_N}.\label{xe33-lz3}
\end{align}
\end{lz}
\pf It only needs to apply  Theorem \ref{mainthm}  to
$
f(x)=x^N(Ax,A/x;q,p)_N,
$
since $f(x)\in \mathcal{L}_N(P(x),Q(x))$ as indicated in Remark \ref{remark1.8} (also see  Lemma \ref{lemma2}).
By making use of the relations \eqref{added-1-1} and \eqref{added-2}, we deduce from \eqref{xe33-3}  that
\begin{align*}
&(Ax,A/x;q,p)_N=\frac{(Cxq,Cq/x;q,p)_{N}}{(q,C^2q;q,p)_N}\\
&\quad\times\sum_{k=0}^{N}\frac{\theta(C^2q^{2k};p)} {\theta(C^2;p)}\frac{(C^2,C/x,Cx,q^{-N};q,p)_{k}} {(q,Cxq,Cq/x,C^2q^{N+1};q,p)_{k}}(ACq^k,Aq^{-k}/C;q,p)_Nq^{k+kN}\\
&\quad=\frac{(Cxq,Cq/x,AC,A/C;q,p)_{N}}{(q,C^2q;q,p)_N}\\
&\quad\times\sum_{k=0}^{N}\frac{\theta(C^2q^{2k};p)} {\theta(C^2;p)}\frac{(C^2,C/x,Cx,q^{-N};q,p)_{k}} {(q,Cxq,Cq/x,C^2q^{N+1};q,p)_{k}}\frac{(ACq^N,Cq/A;q,p)_k}{(Cq^{-N+1}/A,AC;q,p)_k}q^{k},
\end{align*}
which  coincides with  \eqref{xe33-lz3}.
\qed

With the help of  Theorem \ref{mainthm}, we can extend the elliptic Karlsson-Minton type identity \cite[Corollary 2.8]{schlosser}  given by Schlosser and Yoo to the following
\begin{dl}[Generalized  elliptic Karlsson-Minton type identity] \label{karlsson-type} For any $f(x)\in \mathcal{L}_{N_0}(P(x),Q(x))$, we have the expansion
\begin{align}&\bigg(\frac{C}{x}\bigg)^{N_0}\frac{(q,C^2q;q,p)_{m+N_0}}{(Cxq,Cq/x;q,p)_{m+N_0}}~f(x)
\prod_{i=1}^m\theta(A_ix,A_i/x;p)\label{new-final}\\
&=\sum_{k=0}^{m+N_0}q^{k(m+1)}\frac{\theta(C^2q^{2k};p)} {\theta(C^2;p)}\frac{(C^2,C/x,Cx,q^{-(m+N_0)};q,p)_{k}} {(q,Cxq,Cq/x,C^2q^{m+N_0+1};q,p)_{k}}f(Cq^k)\prod_{i=1}^m\theta(A_iCq^k,A_iq^{-k}/C;p).\nonumber
\end{align}
\end{dl}
\begin{proof} To establish \eqref{new-final}, we only need to apply Theorem \ref{mainthm} to the elliptic Askey-Wilson polynomial
\begin{align*}
F(x)=f(x)\prod_{i=1}^mx\theta(A_ix,A_i/x;p).
\end{align*}
Note that $F(x)$ is of degree $m+N_0$. The expansion corresponding to \eqref{xe33-3}  becomes
\begin{align*}&\bigg(\frac{C}{x}\bigg)^{m+N_0}\frac{(q,C^2q;q,p)_{m+N_0}}{(Cxq,Cq/x;q,p)_{m+N_0}}
~f(x)\prod_{i=1}^mx\theta(A_ix,A_i/x;p)\\
&=\sum_{k=0}^{m+N_0}\frac{\theta(C^2q^{2k};p)} {\theta(C^2;p)}\frac{(C^2,C/x,Cx,q^{-(m+N_0)};q,p)_{k}} {(q,Cxq,Cq/x,C^2q^{m+N_0+1};q,p)_{k}}q^{k}~ f(Cq^k)\big(Cq^k\big)^{m}\prod_{i=1}^m\theta(A_iCq^k,A_iq^{-k}/C;p).
\end{align*}
A bit simplification leads us to \eqref{new-final}.
\end{proof}
Letting  $N_0=0$ and $f(x)=1$ in Theorem \ref{karlsson-type}, we immediately obtain
\begin{tl}[The elliptic Karlsson-Minton type identity. Cf. {\rm \cite[Corollary 2.8]{schlosser}}] For any integer  $m\geq 0$, there holds
\begin{align}&\frac{(q,C^2q;q,p)_{m}}{(Cxq,Cq/x;q,p)_{m}}\prod_{i=1}^m\theta(A_ix,A_i/x;p)\label{new-final-final}\\
&=\sum_{k=0}^{m}q^{k(m+1)}\frac{\theta(C^2q^{2k};p)} {\theta(C^2;p)}\frac{(C^2,C/x,Cx,q^{-m};q,p)_{k}} {(q,Cxq,Cq/x,C^2q^{m+1};q,p)_{k}}\prod_{i=1}^m\theta(A_iCq^k,A_iq^{-k}/C;p).\nonumber
\end{align}
\end{tl}
Next is a direct consequence of Theorem \ref{karlsson-type} with the choice  $f(x)=Q(x)^{N_0}$.
\begin{tl} For any integers $m, N_0\geq 0$, there holds
\begin{align}\frac{(q,C^2q;q,p)_{m+N_0}}{(Cxq,Cq/x;q,p)_{m+N_0}}
\prod_{i=1}^m\theta(A_ix,A_i/x;p)\nonumber&\\
=\sum_{k=0}^{m+N_0}q^{k(m+N_0+1)}\frac{\theta(C^2q^{2k};p)} {\theta(C^2;p)}\left(\frac{\theta(-pC^2q^{2k};p^2)}{\theta(-px^2;p^2)}\right)^{N_0}&\label{new-final-fianl}\\
\times\frac{(C^2,C/x,Cx,q^{-(m+N_0)};q,p)_{k}} {(q,Cxq,Cq/x,C^2q^{m+N_0+1};q,p)_{k}}\prod_{i=1}^m\theta(A_iCq^k,A_iq^{-k}/C;p).&\nonumber
\end{align}
\end{tl}

We conclude our paper with the following elliptic analogue of  Gasper's summation formula for VWP ${}_{6+2m}\phi_{5+2m}$ series. The reader may consult  \cite{gasper} or \cite[Exercise  2.33(i)]{10} for details and  Rosengren and  Warnaar's survey \cite[Eq. (1.3.7)]{warnaar} for its multivariate version.
\begin{tl} For any integers $m\geq 1, N_i\geq 0, N=\sum_{i=1}^{m}N_i$, there holds
\begin{align}
&{}_{8+2m}V_{7+2m}(C^2;C/x,Cx,q^{-N},A_1Cq^{N_1},\ldots,A_mCq^{N_m},Cq/A_1,\ldots,Cq/A_m;q,p)\nonumber\\
&\quad=\frac{(q,C^2q;q,p)_{N}}{(Cxq,Cq/x;q,p)_{N}}
\prod_{i=1}^m\frac{(A_ix,A_i/x;q,p)_{N_i}}{(A_iC,A_i/C;q,p)_{N_i}}.\label{xe33-tl3}
\end{align}
\end{tl}
\pf It is a direct consequence of Theorem \ref{generialized} by taking   $f(x)=1$ (i.e., $N_0=0$) in  \eqref{xe33-3-3}.
\qed

\end{document}